\documentclass[12pt,twoside]{article}
 
\usepackage[margin=1in]{geometry} 
\usepackage{amsmath,amsthm,amssymb}
\usepackage{amsfonts}
\usepackage{hyperref}
\usepackage{fancyhdr}
\usepackage{graphicx}
\usepackage{tikz}
\usepackage{float}
\usepackage{indentfirst}
\usepackage{mathtools}

\usetikzlibrary{decorations.pathreplacing}
\newcommand{\R}{\mathbb{R}}
\newcommand{\MF}{\mathcal{M}_\mathcal{F}}
\newcommand{\LF}{\mathcal{L}_\mathcal{F}}
\newenvironment{theorem}[2][Theorem]{\begin{trivlist}
\item[\hskip \labelsep {\bfseries #1}\hskip \labelsep {\bfseries #2.}]}{\end{trivlist}}
\newenvironment{lemma}[2][Lemma]{\begin{trivlist}
\item[\hskip \labelsep {\bfseries #1}\hskip \labelsep {\bfseries #2.}]}{\end{trivlist}}

\newenvironment{corollary}[2][Corollary]{\begin{trivlist}
\item[\hskip \labelsep {\bfseries #1}\hskip \labelsep {\bfseries #2.}]}{\end{trivlist}}
\newenvironment{definition}[2][Definition]{\begin{trivlist}
\item[\hskip \labelsep {\bfseries #1}\hskip \labelsep {\bfseries #2.}]}{\end{trivlist}}

\newtheorem{op}{Open Problem}[section]
\begin{document}
 
\newcommand{\keywords}[1]{%
  \begin{flushleft}
    \textbf{Keywords:} \small #1
  \end{flushleft}
}
\title{Geometric and Analytic Aspects of Simon-Łojasiewicz Inequalities on Vector Bundles} 
\author{
  Owen Drummond \\
  Department of Mathematics, Rutgers University \\
  \texttt{owen.drummond@rutgers.edu} \\
  \\
  \small{Advised by: Natasa Sesum} \\
}
\maketitle
\begin{abstract}
    In real analysis, the Łojasiewicz inequalities, revitalized by Leon Simon in his pioneering work on singularities of energy minimizing maps, have proven to be monumental in differential geometry, geometric measure theory, and variational problems. These inequalities provide specific growth and stability conditions for prescribed real-analytic functions, and have found applications to gradient flows, gradient systems, and as explicated in this paper, vector bundles over compact Riemannian manifolds. In this work, we outline the theory of functionals and variational problems over vector bundles, explore applications to arbitrary real-analytic functionals, and describe the energy functional on $S^{n-1}$ as a functional over a vector bundle. 
\end{abstract}
\tableofcontents
\section{Introduction}
The study of Simon-Łojasiewicz inequalities within the context of functionals defined on vector bundles over compact Riemannian manifolds represents a significant extension of classical analytical techniques, specifically within the calculus of variations and partial differential equations. These inequalities, originally developed by Stanisław Łojasiewicz and later extended by Leon Simon, have found wide applications in the analysis of gradient flows, convergence properties of solutions to elliptic and parabolic equations, and the stability of critical points. A long standing yet recondite problem in the geometric flows and geometric analysis pertains to the uniqueness of tangent maps, and these inequalities have found utility in addressing this question by demonstrating that all sufficiently small perturbations converge back to the same critical point, a condition sufficient for uniqueness.  

With his ingenuity, Simon discovered this application of Łojasiewicz inequalities to energy minimizing maps. We recall the following definition of the Dirichlet Energy functional. Let $u \in W^{1,2}(\Omega; N)$ and let $B_\rho(y)$ be a ball of radius $\rho$ centered at $y \in \Omega$, with $\Bar{B}_\rho(y) \subset \Omega$. The \textit{energy} of $u$ in $B_\rho(y)$ is given by
    \[
    \mathcal{E}_{B_\rho (y)}(u)=\int_{B_\rho (y)} |Du|^2
    \]
    where $Du$ is the $n \times p$ matrix with entries $D_i u^j$, and $|Du|^2 \sum_{i=1}^n \sum_{j=1}^p (D_i u^j)^2$. 
Furthermore, let  $u \in W^{1,2}(B_{\rho}(y); N)$ be a map. If for each ball $B_\rho (y) \subset \Omega$
\[
\mathcal{E}_{B_\rho (y)}(u) \leq \mathcal{E}_{B_\rho (y)}(w) 
\]
for every $w \in W^{1,2}(B_\rho (y); N)$ with $w=u$ in a neighborhood of $\partial B_{\rho}(y)$, the $u$ is called an \textbf{energy minimizing map onto} $\mathbf{N}$.
In this paper, we will delve into the specific application of these inequalities for both arbitrary real-analytic functional, as well as the energy function on $S^{n-1}$ in particular. By considering vector bundles over compact Riemannian manifolds, we introduce a framework by which these inequalities can be utilized in this critcal context for energy minimizers. Through this exploration, the paper aims to illuminate the deep interconnections between differential geometry, the calculus of variations, and geometric analysis, highlighting the versatility and depth of the Simon-Łojasiewicz inequalities in a broader mathematical landscape. We will primarily be referencing \textit{Theorems on Regularity and Singularity of Energy Minimizing Maps}\cite{simon1996theorems} by Leon Simon.
\section{Analytic Preliminaries}
Here, we will assume the reader is familiar with the basic theory of energy minimizing maps, tangent maps, and the singular set. A further exposition on these topics can be found in my paper \textit{Geometric Analysis of Energy Minimizing Maps: Tangent Maps and Singularities}\cite{drummond2024geo}. While not necessary, it would also be useful if the reader has had an introduction to PDEs and the Calculus of Variations. 
\subsection{Variational Equations for Energy Minimizers}
First, we describe the framework that produces a certain class of variational equations of energy minimizers. Let $\Omega$ be an open subset of $\mathbb{R}^n$, $n \geq 2$, and define $N$ to be a smooth, compact Riemannian Manifold of dimension $m \geq 2$ which is isometrically embedded in some Euclidean space $\mathbb{R}^p$. Assume $u \in W^{1,2}(\Omega; N)$ is energy minimizing. 
Let $\Bar{B}_\rho (y) \subset \Omega$, and for some $\delta > 0$, we have the 1-parameter family of maps
\[
\{u_s\}_{s \in (-\delta,\delta)}: B_\rho(y) \rightarrow N
\]
such that $u_0=u$. Moreover, $Du_s \in L^2(\Omega)$, and $u_s \equiv u$ in a neighborhood of $\partial B_\rho(y)$ for each $s \in (-\delta,\delta)$. Be definition, we know that $\mathcal{E}_{B_\rho(y)}(u_s)$ takes its minimum at $s=0$, and hence
\[
\frac{d \mathcal{E}_{B_\rho(y)}(u_s)}{ds}|_{s=0}=0
\]
Of course, this presupposes that the derivative is well defined at $s=0$, and in fact, this derivative is called the \textit{first variation} of $\mathcal{E}_{B_\rho(y)}$ relative to the family $\{u_s\}_{s \in (-\delta,\delta)}$. 

Before stating the first class of variations, we must first define an intuitively simply function which is unfortunately fraught with much technical baggage that will not be covered here, but for a further exposition, I recommend Section 2.12.3 in Simon's book (p.42-46)\cite{simon1996theorems}.
\begin{definition}{2.1 (Nearest Point Projection)}
Given a compact $C^\infty$ manifold $N$, and a tubular neighborhood $U=\{x \in \mathbb{R}^n: \text{dist}(x,N) < \delta \}$ of $N$, the \textbf{nearest point projection} is a map $\Pi \in C^{\infty}(\{x: \text{dist}(x,N) < \delta \}: \mathbb{R}^p)$ that takes a point $x \in U$ to the nearest point of $N$.
\end{definition}
\begin{figure}[H] 
    \centering
    \includegraphics[width=0.33\linewidth]{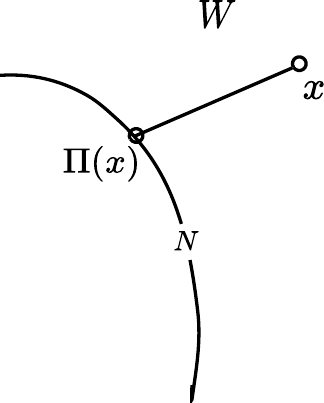}
    \caption{\textit{Nearest Point Projection}}
    \label{fig:my_label}
\end{figure}
Given this, we define the \textit{first class of variations} of energy minimizers as 
\[
u_s=\Pi \circ (u+s \xi)
\]
where $\Pi$ is the nearest point projection onto $N$, and $\xi=(\xi^1, \xi^2, ..., \xi^p)$ with each $\xi^j \in C_c^\infty (B_\rho (y))$. We have that the projection $\Pi$ onto $N$ is well defined in some open subset $W \supset N$, and so $u_s$ is defined as an admissible variation for $|s|$ sufficiently small. Here, we apply $D_i$ to the Taylor polynomial expansion of $\Pi$, we see that
\[
D_i u_s=D_i u +s(d \Pi_u(D_i \xi) + \text{Hess} \Pi_u(\xi, D_i u))+E
\]
where $\lvert E \rvert_{L^1(B_\rho(y))} \leq Cs^2$ for $|s|$ small. Now, by plugging this expansion in for the energy $\mathcal{E}_{B_\rho(y)} (u_s)$, and by leveraging information about the Hessian given of $\Pi$, we have the identity
\[
\int_{\Omega} \sum_{i=1}^n (D_i u \cdot D_i \xi - \xi \cdot A_u (D_i u, D_i u))=0
\]
where $A_u$ is the second fundamental form of $N$ at $u(x)$. Note that since $\xi$ is smooth and compactly supported, the if $u$ is $C^2$ then we can integrate by parts to obtain the equation
\[
\Delta u + \sum_{i=1}^n A_u (D_i u, D_i u) =0
\]
where $\Delta u = (\Delta u^1, ..., \Delta u^p)$. This is the main identity for $u \in C^2$ that we will utilize later.
\subsection{Łojasiewicz Inequalities}
Here, we will present the two inequalities attributed to Polish mathematician Stanisław Łojasiewicz\cite{lojasiewicz1965}.
\begin{theorem}{2.2.1 (Łojasciewz 1965)}
    Let $\Omega$ be an open subset of $\mathbb{R}^n$, and $f:\Omega \rightarrow \mathbb{R}$ a real-analytic function. Denote $Z$ as the set of all zeros of $f$, and for every compact subset $K$ of $\Omega$, $\exists \alpha, C > 0$ such that
    \[
    |\text{dist}(x, Z)|^{\alpha} \leq C|f(x)| \hspace{0.5 cm} \forall x \in K
    \]
\end{theorem}
This inequality states that the (smallest) distance between a point in a compact subset $K \subset \Omega$ never exceeds the magnitude of $f$ itself, for appropriate constants $\alpha, C$.
\begin{theorem}{2.2.2 (Łojasciewz)}
Let $\Omega$ and $f$ be as above. For every critical point $x \in \Omega$ of $f$, $\exists$ a neighborhood $V$ of $x$, and an exponent $\theta \in [1/2,1]$ and a constant $C>0$ such that
\[
|f(x)-f(y)|^\theta \leq C|\nabla f(y)| \hspace{0.5 cm} \forall y \in V
\]
Thus, we can always provide upper bound, given by the magnitude of the gradient of $f$ for the distance between any two points in a suitable neighborhood of $x$, given appropriate constants $\theta$ and $C$.
\end{theorem}
It is worth mentioning that these inequalities have profound implications in the theory of ODEs and PDEs. As seen in a paper of Ralph Chill, one application is the \textit{Simon-Łojasciewz Gradient Inequality}\cite{chill2003lojasiewicz}, which is concerned with energy functionals defined on some Hilbert Space $V \hookrightarrow L^2(\Omega)$. Chill proved several instances where these inequalities still held regardless of analyticity assumptions. Further applications include minimal networks, geometric flows, and gradient systems.

\section{Calculus of Variations on Vector Bundles}

\subsection{Notation}
The energy functional provides motivation for explicating a general theory of (smooth) functionals defined on smooth sections of a vector bundle over a compact Riemannian manifold. For example, if we wish to study the structure of tangent maps and $\text{sing}u$ for some energy minimizing map $u \in W^{1,2}(S^{n-1};N)$, then this general framework provides a way to describe the properties of $\mathcal{E}_{S^{n-1}}(u)$ by examining its behavior over vector bundles of $S^{n-1}$. The following is the general framework:
Let $n<p_1$, $q \leq p_2$ be positive integers, and let $\Sigma$ be an n-dimensional compact $C^1$ Riemannian manifold isometrically embedded in $\mathbb{R}^{p_1}$. Define the vector bundle $\mathbf{V}=\cup_{\omega \in \Sigma} V_\omega$, where each $V_\omega$ is some q-dimensional subspace of $\mathbb{R}^{p_2}$. Here, $V_\omega$ varies smoothly with respect to $\omega$, that is the matrix of the orthogonal projection $P_\omega$ of $\mathbb{R}^{p_2}$ onto $V_\omega$ is a $C^\infty$ function of $\omega$.
\begin{figure}[H] 
    \centering
    \includegraphics[width=0.75\linewidth]{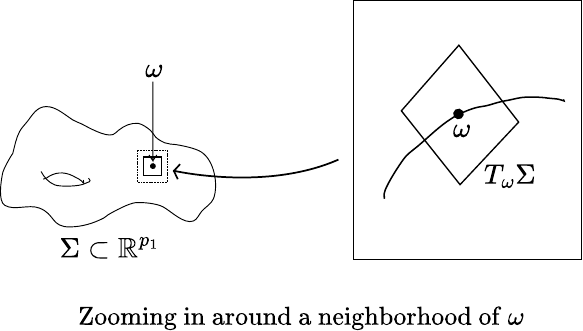}
    \caption{The tangent spaces $T_\omega \Sigma$}
    \label{fig:my_label}
\end{figure}
\begin{definition}{3.1.1 (Sections)}
For $k=0,1,...,C^k(\mathbf{V})$, we define the set of $\mathbf{C^k}$ \textbf{sections} of $\mathbf{V}$ to be maps $u \in C^k(\Sigma; \mathbb{R}^{p_2})$ with $u(\omega) \in V_\omega$ for each $\omega \in \Sigma$. Moreover, for any $\alpha \in (0,1]$, $C^{k,\alpha}(\mathbf{V})$ denotes the $\mathbf{C^{k,\alpha}}$ \textbf{sections of} $\mathbf{V}$, that is, functions that are Hölder continuous with exponent $\alpha$. 
\end{definition}
Further, $L^2(\mathbf{V})$ denotes the subspace of $L^2(\Sigma; \mathbb{R}^{p_2})$ such that $u(\omega) \in V_\omega$ for a.e. $\omega \in \Sigma$. Lastly, we define the $L^2$ inner product as
\[
\langle u,v \rangle_{L^2(\mathbf{V})} = \int_\Sigma u(\omega) \cdot v(\omega) d\omega
\]
\textbf{Note}: $d\omega$ is the Riemannian volume form associated with the metric on $\Sigma$.
\subsection{Gradient Operators}
Here, we would like to make sense of the gradient operator over vector bundles (and manifolds) in relation to the classical Euclidean directional derivative. First, let $\tau_1,...,\tau_n$ be an orthonormal basis for $T_\omega \Sigma$ of locally defined smoothly varying $\mathbb{R}^{p_1}$-valued functions $\Sigma$. We denote $\nabla_{\tau_i}$ as the (Euclidean) derivative in the direction of $\tau_i$, and now define
\[
\nabla_{\tau_i}^{\mathbf{V}}u=P_{\omega}(\nabla_{\tau_i})
\]
where $P_\omega$ is the orthogonal projection onto $\omega$. From this we can now make the following definition:
\begin{definition}{3.2.1 (Gradient Operator on the Vector Bundle $\mathbf{V}$)}
Given an orthonormal basis $\tau_1,...,\tau_n$ of $T_\omega \Sigma$, for any $u \in C^1(\mathbf{V})$, we define the \textbf{gradient over a vector bundle} to be 
\[
\nabla^{\mathbf{V}}=\sum_{i=1}^n \tau_i \otimes \nabla_{\tau_i}^{\mathbf{V}}u
\]
where $\otimes$ is the typical tenor product. Note that induces the map $\nabla^{\mathbf{V}}: C^1(\mathbf{V}) \rightarrow T_\omega \Sigma \otimes V_w \subset \mathbb{R}^{p_1} \otimes \mathbb{R}^{p_2} \cong \mathbb{R}^{p_1 p_2}$. This assumes the typical identification of $\mathbb{R}^{p_1} \otimes \mathbb{R}^{p_2}$ with $\mathbb{R}^{p_1 p_2}$ induced by the map 
\[
(x^1,...,x^{p_1}) \otimes (y^1,...,y^{p_2}) \mapsto (x^i y^j)_{i=1,...,p_1, j=1,...,p_2}
\]

\end{definition}
From the we obtain the immediate corollary:
\begin{corollary}{3.2.2}
    The gradient operator $\nabla^{\mathbf{V}}: C^1(\mathbf{V}) \rightarrow \mathbb{R}^{p_1 p_2}$ is globally defined on $\Sigma$, and is independent of the choice of basis $\tau_1,...,\tau_n$.
\end{corollary}
\begin{proof}
As $u \in C^k(\mathbf{V})$, we can write $u=\sum_{i=1}^{p_2} u^j e_j$, where $e_j$ is the standard basis vector. Now, we see that $\nabla_{\tau_i} u = \sum_{i=1}^{n} \tau_i (u^j) e_j$. Also, the gradient operator over the manifold $\Sigma$ is given by $\nabla^\Sigma u^j = \sum_{i=1}^n \tau_i (u^j) \tau_i$. In the definition of $\nabla^{\mathbf{V}}$, we have
\[
\nabla^{\mathbf{V}}=\sum_{i=1}^n \tau_i \otimes P_{\omega}(\nabla_{\tau_i}) = \sum_{j=1}^{p_2} (\nabla^{\Sigma} u^j) \otimes P_\omega (e_j)
\]
which is independent of $\tau_1,...,\tau_n$.
\end{proof}

\subsection{The Euler-Lagrange Operator $\mathcal{M}_\mathcal{F}$}
Next, for the sake of using techniques from the calculus of variations, we would like to define a functional over the manifold with $u$-dependence, where $u \in C^1(\mathbf{V})$ given by
\[
\mathcal{F}(u)=\int_{\Sigma} F(\omega, u, \nabla^\mathbf{V} u)
\]
where $F=F(\omega, Q)$ is smooth real-valued function, and $\omega \in \Sigma$, $Q \in \mathbb{R}^{p_1} \times \mathbb{R}^{p_1 p_2}$. Next, we make the following definition which will lead us to the \textit{first variation} of $\mathcal{F}$:
\begin{definition}{3.3.1 (The Euler-Lagrange Operator $\MF$)}
The Euler-Langrange operator $\MF$ for $\mathcal{F}$ is defined on $C^2(\mathbf{V})$ by the following requirements. First, $\MF(u) \in C^0(\mathbf{V})$, and secondly,
\[
{\frac{d}{ds} \mathcal{F}(u+sv) }|_{s=0}=\langle \MF(u), v \rangle_{L^2(\Sigma)}, \hspace{0.5cm} u,v \in C^2(\mathbf{V})
\]
\end{definition}
As we see, this definition does not uniquely determine $\MF$. Using corollary 3.2.2, we can express $\nabla^{\mathbf{V}}$ independent of the choice the basis by instead using the global gradient operator and setting
\[
F(\omega, u, \nabla^{\mathbf{V}})=\Tilde{F}(\omega,u,\nabla^{\Sigma} u^1,..., \nabla^{\Sigma} u^{p_2})
\]
Here, we have $\Tilde{F}(\omega,z,\eta)$, where $z \in \mathbb{R}^q$, and $\eta=(\eta^{(1)},...,\eta^{(p_2)})$, with $\eta^{(\alpha)}=(\eta_1^{\alpha},...,\eta_{p_1}^{\alpha}) \in \R^{p_1}$, and we see that this defines the relation 
\[
\Tilde{F}(\omega, z, \eta^{(1)}, ..., \eta^{(p_2)})=F(\omega,z,\sum_{j=1}^{p_2} \eta^{(j)} \otimes  P_\omega(e_j))
\]
Now we can evaluate the first variation $\frac{d}{ds} \mathcal{F}(u+sv)|_{s=0}$ directly by first noting
\begin{align*}
    \frac{d}{ds} \mathcal{F}(u+sv)|_{s=0} = \int_{\Sigma} F(\omega, u+sv, \nabla^{\Sigma}(u+sv))|_{s=0} d\omega \\ 
    =\int \frac{d}{ds} \Tilde{F}(\omega,u,\nabla^{\Sigma} u^1,..., \nabla^{\Sigma} u^{p_2})_{s=0}
\end{align*}
now by chain rule:
\begin{align*}
    \frac{d}{ds} \Tilde{F}
    =\frac{\partial \Tilde{F}}{\partial u} \frac{d}{ds} u + \sum_{j=1}^{p_1} \frac{\partial \Tilde{F}}{\partial (\nabla^{\Sigma} u^j)} \cdot \frac{d}{ds}\nabla^{\Sigma} u^j \\
    =\frac{\partial \Tilde{F}}{\partial u} v + \sum_{j=1}^{p_1} \frac{\partial \Tilde{F}}{\partial (\nabla^{\Sigma} u^j)} \cdot \nabla^{\Sigma} v^j
\end{align*}
by identifying $\frac{d}{ds} u=v$ and $\frac{d}{ds}\nabla^{\Sigma} u^j = v^j$. Now summing over $\alpha=1$ to $p_2$ and substituting back into the integral, we obtain that 
\[
\frac{d}{ds} \mathcal{F}(u+sv)|_{s=0} = \int_{\Sigma} \sum_{\alpha=1}^{p_2}(\frac{\partial \Tilde{F}}{\partial u^{\alpha}} v^{\alpha} + \sum_{j=1}^{p_1} (\frac{\partial \Tilde{F}}{\partial (\nabla_j u^{\alpha})} \partial_j v^{\alpha}))
\]
where $\nabla_j=e_j \cdot \nabla^{\Sigma}$, $j=1,...,p_1$. Finally, this leads to the equation
\[
\frac{d}{ds} \mathcal{F}(u + sv) |_{s=0} = \int_{\Sigma} \sum_{\alpha=1}^{p_2} ( \sum_{j=1}^{p_1} (\nabla_j v^\alpha) F_{\eta_j}^\alpha (\omega, \tilde{u}, \nabla \tilde{u}) + v^\alpha F_{z_\alpha}(\omega, \tilde{u}, \nabla \tilde{u})) d\omega
\]
by identifying $\Tilde{F}_{\eta_j}^{\alpha}$ and $\Tilde{F}_{z_\alpha}$ with the partial derivatives of $\Tilde{F}$ with the components of $\nabla u$ and $u$ respectively. Now we can employ the integration formula
\[
\int_{\Sigma} \nabla_j f = -\int_{\Sigma} f H_j
\]
on $\Sigma$, where $f \in C^\infty(\Sigma)$ and $(H_1,...,H_{p_1})$ is the mean-curvature vector of $\Sigma$. Since $\Sigma$ can be thought of as a hypersurface embedded in $\mathbb{R}^{p_1}$, if $\text{II}$ is the second fundamental form of $\Sigma$, and $N$ is the unit normal vector, we can write
\[
H=\text{tr}(\text{II})N=\left(\sum_{i=1}^{p_1-1} \text{II}(e_i,e_i) \right)N
\]
where $\text{II}(X,Y)=-\langle \nabla_X N, Y \rangle$, and $X,Y$ are tangent vectors on $\Sigma$. With this, we now see that
\[
\left. \frac{d}{ds} \mathcal{F}(u + sv) \right|_{s=0} = \int_{\Sigma} \sum_{\alpha=1}^{p_2} v^\alpha \left( -\sum_{j=1}^{p_1} \nabla_j \tilde{F}_{\eta_j}^{(\alpha)}(\omega, \tilde{u}, \nabla \tilde{u}) + H_j \hat{F}_{\eta_j}^{(\alpha)}(\omega, \tilde{u}, \nabla \tilde{u}) + \tilde{F}_{z_\alpha}(\omega, \tilde{u}, \nabla \tilde{u}) \right) \, d\omega
\]
Now comparing this equation with the relation
\[
\frac{d}{ds} \mathcal{F}(u+sv)|_{s=0}=\langle \MF, v \rangle_{L^2(\Sigma)}=\int_{\Sigma} \MF \cdot v d\omega
\]
and setting $v^\alpha=e_\alpha$, we obtain the identity
\[
e_\alpha \cdot M_{\mathcal{F}}(u) = \sum_{j=1}^{p_1} \nabla_j \left(F_{\eta_j}^{(\alpha)}(\omega, \tilde{u}, \nabla \tilde{u})\right) - F_{z_\alpha}(\omega, \tilde{u}, \nabla \tilde{u}) - \sum_{j=1}^{p_1} H_j F_{\eta_j}^{(\alpha)}(\omega, \tilde{u}, \nabla \tilde{u})
\]
and hence $\MF$ exists and is described by this relation for each $\alpha=1,...,p_2$. Notice that by a rearrangement, this takes the general form
\[
e_\alpha \cdot M_{\mathcal{F}}(u) = \sum_{j=1}^{p_1} \sum_{\beta=1}^{p_2} F_{\eta_j}^{(\alpha)}(\omega, \tilde{u}, \nabla \tilde{u}) \nabla_j v^\beta - f_\alpha(\omega, \tilde{u}, \nabla \tilde{u})
\]
where $f=f(\omega, z, \eta)$ is a smooth function of $(\omega, z, \eta) \in \Sigma \times \mathbb{R}^p_2 \times \mathbb{R}^{p_1 p_2}$. 
Here, we will always assume that $F$ is defined such that the operator $\MF$ is \textit{elliptic}, that is
\[
\sum_{i,j=1}^{p_1} \sum_{\alpha,\beta=1}^{p_2} \tilde{F}_{{\eta_i}^{(\alpha)}, {\eta_i}^{(\beta)}}(\omega, z, \eta) \lambda_\alpha \lambda_\beta \xi^i \xi^j > 0
\]
for every $\xi=(\xi^1,...,\xi^{p_1}) \in T_\omega \Sigma \setminus \{0 \}$ and for all $\lambda=(\lambda_1,...,\lambda_{p_2}) \in V_\omega \setminus \{0 \}$.

\subsection{The Linearization $\mathcal{L}_\mathcal{F}$}

Now we define the linearization, or \textit{second variation} of the Euler-Lagrange operator $\MF$. 
\begin{definition}{3.4.1 (The Second Variation)}
Assume that $u \in C^2(\mathbf{V})$ is a solution of $\MF(u)=0$ on $\Sigma$, and we define the linearized operator $\mathcal{L}_{\mathcal{F},u}$ of $\MF$ at $u$ by
\[
\mathcal{L}_{\mathcal{F},u}(v)=\frac{d}{ds} \MF(u+sv)|_{s=0}, \hspace{0.5cm} v \in C^2(\mathbf{V})
\]
\end{definition}
Now let $\sigma_0 > 0$, and assume $u \in C^2(\mathbf{V})$ with $|u|_{C^2} < \sigma_0$, and $u_1,u_2 \in C^2(\mathbf{V})$ are arbitrary with the condition that $|u_1|_{C^2},|u_2|_{C^2} \leq \sigma_0$. Note that we can write $\MF(u_j) \equiv \MF (u+(u_j-u))$ for $j=1,2$, and using the following identity using Taylor expansion:
\[
f(1)=f(0)+f'(0)+\int_{0}^1 (1-s) f''(s) ds
\]
and setting $f(s)=\MF(u+s(u_j-u))$, we obtain
\[
\MF(u_j)-\MF=\mathcal{L}_{\mathcal{F},u}(u_j-u)+N(u,u_j)
\]
where
\[
N(u,u_j)=\int_{0}^1 (1-s) \frac{d^2}{ds^2} \MF(u+s(u_j-u)) ds, \hspace{0.5cm} j=1,2
\]
Now using the fact that $|u+s(u_j-u)|_{C^2} \leq \sigma_0$ for all $s \in [0,1]$, we deduce the identity
\[
\MF(u_1)-\MF(u_2)=\mathcal{L}_{\mathcal{F},u}(u_1-u_2)+a\cdot \nabla^2(u_1-u_2)+b\cdot \nabla (u_1-u_2)+c\dot (u_1-u_2)
\]
where
\[
\sup(|a|+|b|+|c|) \leq C(|u_1-u|_{C^2}+|u_2-u|_{C^2})
\]
and $C$ depends only on $\mathcal{F}$ and the choice of $\sigma_0$.
\section{The Liapunov-Schmidt Reduction for $\mathcal{M}_\mathcal{F}$}
Now we turn our attention to the Liapunov-Schmidt reduction for the operator $\MF$. This process involves decomposing the space on which a function acts into orthogonal subspaces, hence rendering non-linear problems in PDEs more manageable. In fact, this reduction proves to be particular effective in the case where the implicit function theorem fails. Essentially, for an arbitrary functions $\mathcal{F}$ defined over a vector bundle, this reduction simplifies infinite-dimensional problems by reducing to a finite setting. 
\subsection{The Operator $\mathcal{N}$}
Here we introduce an operator that induces a kind of orthogonal decomposition of $\MF$ and $\LF$. First we would like to recall the orthogonal projection $P_\Omega:  L^2(\mathbf{V}) \rightarrow \Omega$ for some $\Omega \subset L^2(\mathbf{V})$. For some $v \in L^2(\mathbf{V})$, we define
\[
P_{\Omega}(v)=u \in \Omega \hspace{0.25cm} \text{such that} \hspace{0.25cm} \langle v-u,w \rangle_{L^2(\mathbf{V})}=0 \hspace{0.5cm} \forall w \in \Omega
\]
it is easy to verify that the operator $P_{\Omega}$ is linear, idempotent, and self-adjoint.
\begin{definition}{4.1.1 (The Decomposition $\mathcal{N}$)}
Let $K$ be the finite dimensional kernel of the elliptic operator $\LF$, and let $P_k$ be the orthogonal projection of $L^2(\mathbf{V})$ onto $K$. We define operator $\mathcal{N}: C^{2,\alpha}(\mathbf{V}) \rightarrow C^{0,\alpha}(\mathbf{V})$ to be 
\[
\mathcal{N}u=P_K u+\MF u
\]
\end{definition}
There is an important fact we will now state about $\mathcal{N}$:
\begin{corollary}{4.1.2}
$\mathcal{N}$ is a bijection of a neighborhood $U$ of 0 in $C^{2,\alpha}(\mathbf{V})$ onto a neighborhood $W$ of 0 in $C^{0,\alpha}(\mathbf{V})$.
\end{corollary}
\begin{proof}
    First, $\mathcal{N}(0)=P_K(0)+\MF(0)=0$ as $\MF=0$, and the linearization of $\mathcal{N}$ is given by
    \[
    d\mathcal{N}|_{0}(v) \equiv \frac{d}{ds} \mathcal{N}(sv)|_{s=0}=P_k+\LF
    \]
    Now, solving for $P_k(v)+\LF(v)=0$ yields that $P_k(v)=-\LF(v)$. However, note that $P_k(v) \in K$, and $K=\ker(\LF)$, which would imply that $\LF(v)$ equals the negative part of its kernel in each component, which is a contradiction unless $K=0$. Thus, $\frac{d}{ds} \mathcal{N}(sv)|_{s=0}=P_k+\LF$ has a trivial kernel, and hence $d\mathcal{N}|_{0}$ is an isomorphism of $C^{2,\alpha}(\mathbf{V})$ into $C^{0,\alpha}$ for each $\alpha \in (0,1)$. By the inverse function theorem, which can be applied to the $C^1$ operator $\mathcal{N}:C^{2,\alpha}(\mathbf{V}) \rightarrow C^{0,\alpha}(\mathbf{V})$, we have that $\mathcal{N}$ is a bijection of a neighborhood $U$ of $0$ in $C^{2,\alpha}(\mathbf{V})$ to a neighborhood $W$ of $0$ in $C^{0,\alpha}(\mathbf{V})$, which is what we desired to show.
\end{proof}
Here, we will also define $\Psi=\mathcal{N}^{-1}$ from this neighborhood $W$ onto $U$, and we see that $\Psi$ is a $C^1$ function. From now on, we will assume without loss of generality that
\[
U \subset \{u \in C^{2,\alpha}(\mathbf{V}): \lvert u \rvert_{C^{2, \alpha}} < 1 \}
\]

\subsection{$L^2$ Estimates}

Now onto an important lemma which gives a $L^2$ estimate on $\Psi = \mathcal{N}^{-1}$
\begin{lemma}{4.2.1}
    For a neighborhood $\hat{W} \subset W$ of 0 in $C^{0,\alpha}(\mathbf{V})$, depending on $\mathcal{F}$ alone, we have
    \[
    \lvert \Psi(f_1)-\Psi(f_2) \rvert_{W^{2,2}} \leq C \lvert f_1 - f_2 \rvert_{L^2} \hspace{0.5cm} f-1, f_2 \in W^{2,2}(\mathbf{V})
    \]
    and $C$ depends only on $\mathcal{F}$, and we write the $W^{2,2}$-norm to be
    \[
    \lvert v \rvert_{W^{2,2}}^{2} = \lvert v \rvert_{L^{2}}^{2} + \lvert \nabla v \rvert_{L^{2}}^{2} + \lvert \nabla^2 v \rvert_{L^{2}}^{2}
    \]
\end{lemma}
\textbf{Remark}: We can take $\hat{W} = W$ as it is just a matter of notation.
\begin{proof}
    Set $u_j=\Psi(f_j)$, and note that
    \[
    P_K(u_j)+\MF(u_j)=f_j, \hspace{0.5cm} f_j \in W
    \]
    since $\mathcal{N}\Psi(f_j)=f_j$. Now according to the final results of Section 3.4, and setting $u=\varphi$, we have
    \[
    P_K(u_2-u_1)+\mathcal{L}_{\mathcal{F}, \varphi}(u_1-u_2) = a\cdot \nabla^2(u_1-u_2)+b \cdot \nabla(u_1-u_2)+c \cdot (u_1-u_2)+f_2-f_1    
    \]
    with the condition that
    \[
    \sup(|a|+|b|+|c|) \leq C(|u_1-\varphi|_{C^2}+|u_2-\varphi|_{C^2}
    \]
    Now we take the projection onto $K$ and $K^{\perp}$, and since $\mathcal{L}_{\mathcal{F}, \varphi}$ only takes values in $K^\perp$, we have
    \[
    P_K(u_2-u-1)=P_K(a\cdot \nabla^2(u_1-u_2)+b \cdot \nabla(u_1-u_2)+c \cdot (u_1-u_2)+f_2-f_1)
    \]
    and 
    \[
    \mathcal{L}_{\mathcal{F},\varphi}((u_2-u_1)^\perp)=(a\cdot \nabla^2(u_1-u_2)+b \cdot \nabla(u_1-u_2)+c \cdot (u_1-u_2)+f_2-f_1)^{\perp}
    \]
    Now given that $\mathcal{L}_{\mathcal{F}, \varphi}$ is elliptic, we can appeal to a standard $L^2$ estimate:
    \[
    \lVert (u_1-u_2)^\perp \rVert_{W^{2,2}} \leq C \lVert (a\cdot \nabla^2(u_1-u_2)+b \cdot \nabla(u_1-u_2)+c \cdot (u_1-u_2)+f_2-f_1)^{\perp} \rVert_{L^2}
    \]
    where $C$ depends only on $\mathcal{F}$. Now in light of the inequality $\sup(|a|+|b|+|c|) \leq C(|u_1-\varphi|_{C^2}+|u_2-\varphi|_{C^2})$, we have that
    \[
    \lVert (u_1-u_2)^\perp \rVert_{W^{2,2}} \leq C(|u_1-\varphi|_{C^2}+|u_2-\varphi|_{C^2}) \lVert u_1-u_2 \rVert_{W^{2,2}}+C\lVert f_1-f_2 \rVert_{L^2}
    \]
    Now by taking $L^2$ norms on both sides of (2), and again using (1), we deduce that
    \[
    \lVert P_K(u_1 - u_2) \rVert_{L^2} \leq C \left( \lVert u_1 - \phi(c) + u_2 - \phi(c) \rVert \right) \lVert u_1 - u_2 \rVert_{w,2} + C \lVert f_1 - f_2 \rVert_{L^2}
    \]
    Now since $K$ is finite dimensional and is spanned by an orthonormal set of smooth functions $\varphi_1,...,\varphi_q$, we have that 
    \[
    \lVert P_K(u_1 - u_2) \rVert_{W^{2,2}} \leq C \left( \lVert u_1 - \phi(c) + u_2 - \phi(c) \rVert \right) \lVert u_1 - u_2 \rVert_{w,2} + C \lVert f_1 - f_2 \rVert_{L^2}
    \]
    and by adding (3), and (4), we finally obtain 
    \[
    \lVert u_1-u_2 \rVert_{W^{2,2}} \leq C(|u_1-\varphi|_{C^2}+|u_2-\varphi|_{C^2}) \lVert u_1 - u_2 \rVert_{W^{2,2}}+C \lVert f_1-f_2 \rVert_L^2
    \]
    and if $|u_1-\varphi|_{C^2}$ and $|u_2-\varphi|_C^2$ are chosen to be small enough (depending on the choice of $\mathcal{F}$) to guarantee 
    \[
    C(|u_1-\varphi|_{C^2}+|u_2-\varphi|_{C^2})<\frac{1}{2}
    \]
    then the conclusion follows.
\end{proof}
\subsection{The Kernel of $\mathcal{M}_\mathcal{F}$}
Now we enter into a rather computationally heavy section in hopes of better understanding the structure of the kernel of $\MF$. First, note that for a given basis of smooth functions $\varphi_1,...,\varphi_k$, that by definition
\[
\mathcal{N}(\Psi(\sum_{j=1}^l \xi^j \varphi_j ))=\sum_{j=1}^l \xi^j \varphi_j
\]
and therefore
\[
(\MF(\Psi(\sum_{j=1}^l \xi^j \varphi_j )))^\perp = 0, \hspace{0.5cm} \sum_{j=1}^l \xi^j \varphi_j \in W
\]
This is because the action $\MF(\Psi(w))=w-P_K(w)$, for $w \in W$. Therefore, any component of $\MF(\Psi(w))$ in $K^\perp$ must be 0 since $P_K(\MF(\Psi(w)))=0$. Now using the lemma, we would like to set $f_1=P_K(u)$ and $f_2=\Psi^{-1}(u)$, where $u$ is chosen such that $u \in U$ and $P_k(u) \in W$. We have
\[
\lVert \Psi(P_K(u))-u \rVert_{W^{2,2}} \leq C \lVert P_K(u)-\Psi^{-1}(u) \rVert_{L^2} \equiv \lVert \MF(u) \rVert_{L^2}
\]
by the definition of $\mathcal{N}$. By substituting $P_K u$ in place of $u$, we have 
\begin{align}
\lVert \Psi(P_K(u))-P_K(u) \rVert_{W^{2,2}} \leq C \lVert P_K u \rVert_{L^2}^2, \hspace{0.5cm} u \in U'
\end{align}
where we define $U'=\{u \in U: P_K u \in W \}=U \cap P_K^{-1}(W \cap K)$,since $\LF P_K u=0$ and therefore
\[
\lVert \MF(P_K u) \rVert_{L^\infty} \leq C \lVert P_K u \rVert_{C^2}^2 \leq C \lVert P_K u \rVert_{L^2}^2
\]
Note also that (1) implies that 
\begin{align}
d \Psi|_{0} \circ P_K=P_K
\end{align}
by the definition of linearization and the bound on $\lVert \Psi(P_K(u))-P_K(u) \rVert_{W^{2,2}}$ for any $u \in U'$. 
Now we define the function $f(\xi)=\mathcal{F}(\Psi(\sum_{j=1}^l \xi^j \varphi_j))$, for $\sum_{j=1}^l \xi^j \varphi_j \in W$. Now for $\eta \in \R^q$, one can verify by direct computation that
\[
\langle \eta, \nabla f(\xi) \rangle_{L^2} \equiv \langle \MF(\Psi(\sum_{j=1}^s \xi^j \phi_j)), d\Psi(\sum_{j=1}^s \xi^j \phi_j)(\sum_{j=1}^r \eta^j \phi_j) \rangle_{L^2}
\]
for $\xi \in \hat{W}$, and $\hat{W}$ is the open neighborhood of $0 \in \R^l$ such that $\xi \in \hat{W} \iff \sum_{j=1}^l \xi^j \varphi_j \in W \cap K$. Similarly, by subtracting off and adding a component of $P_K$ within the inner products, we have that
\begin{align*}
\langle \eta, \nabla f(\xi) \rangle_{L^2} \equiv \langle \MF(\Psi(\sum_{j=1}^l \xi^j \varphi_j)), d\Psi(\sum_{j} \xi^j \varphi_j)(\sum_{j} \eta^j \varphi_j) - P_K(\sum_{j=1}^l \eta^j \varphi_j) \rangle_{L^2}\\ + \langle \MF(\Psi(\sum_{j=1}^l \xi^j \varphi_j)), P_K(\sum_{j=1}^l \eta^j \varphi_j) \rangle_{L^2}
\end{align*}
Notice that by (2), we have that
\[
\lVert d\Psi(\sum_{j} \xi^j \varphi_j)(\sum_{j} \eta^j \varphi_j) - P_K(\sum_{j=1}^l \eta^j \varphi_j) \rVert_{L^2} \leq C|\xi|
\]
by also taking $\sum_j \eta^j \varphi_j$ to be parallel to $\MF(\Psi \sum_{j=1}^l \xi^j \varphi_j$, we have that
\begin{align*}
\lVert \MF(\Psi(\sum_{j=1}^l \xi^j \varphi_j)) \rVert \leq |\nabla f(\xi)| + C|\xi| \lVert \MF(\Psi(\sum_{j=1}^l \xi^j \varphi_j)) \rVert \\
\implies |\nabla f(\xi)| \leq (1+C|\xi|) \lVert \MF(\Psi(\sum_{j=1}^l \xi^j \varphi_j)) \rVert
\end{align*}
by taking $\eta$ to be parallel to $\nabla f(\xi)$ as well. We would for $C|\xi| \leq 1$ if necessary, so we take an even smaller neighborhood of $W$ in order for this to be the case, and then we would conclude
\begin{align}
\frac{1}{2}|\nabla f(\xi)| \leq \lVert \MF(\Psi(\sum_{j=1}^l \xi^j \varphi_j)) \rVert \leq 2|\nabla f(\xi)|
\end{align}
where $\xi \in \hat{W}$. Now if $u \in U'$, we have that
\[
\MF(u)=0 \implies \mathcal{N}(u)=P_K u \implies \Psi(\mathcal{N})(u)=\Psi(P_K u) \implies u=\Psi(P_K u)
\]
Note that $\nabla f(\xi)=0$ with $\xi \in \hat{W}$ such that $\sum_j \xi^j \varphi_j = P_K u$. Thus, if $\xi \in \hat{W}$ with $\nabla f(\xi)=0$, then $\MF(\Psi(\sum_j \xi^j \varphi_j))=0 \implies P_K(\Psi(\sum_j \xi^j \varphi_j))=\mathcal{N}(\Psi(\sum_j \xi^j \varphi_j))=\sum_j \xi^j \varphi_j$. This implies that $\Psi(\sum_j \xi^j \varphi_j) \in U'$ if $\xi \in \hat{W}$. This finally yields a description of the \textbf{kernel of }$\mathbf{\MF}$ inside $U'$, given by 
\[
\ker_{U'}(\MF)= \{u \in U': \MF u = 0 \} = \Psi(\{\sum_{j=1}^l \xi^j \varphi_j : \xi \in \hat{W}, \nabla f(\xi)=0 \})
\]
since $\Psi: C^{0,\alpha} \rightarrow C^{2,\alpha}$ is a $C^1$ diffeomorphism that embeds $W \cap K$ into $u$, we obtain the immediate corollary:
\begin{corollary}{4.3.1}
    The set
    \[
    M := \{\Psi (\sum_{j=1}^l \xi^j \varphi_j): \xi \in \hat{W} \}
    \]
    is $C^1$ submanifold of $U$ of dimension $l$ that contains the entire kernel of $\MF$ in the neighborhood $U'$.
\end{corollary}
\begin{figure}[H] 
    \centering
    \includegraphics[width=0.33\linewidth]{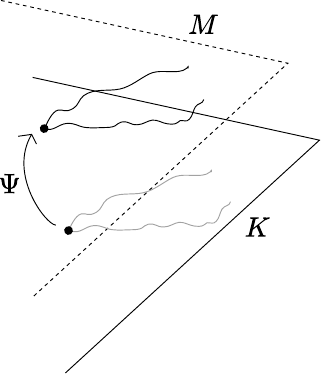}
    \caption{The Liapunov-Schmidt Reduction}
    \label{fig:my_label}
\end{figure}
\subsection{Approximations for $\mathcal{F}$}
Here, recall the previous definition of $f$:
\[
f(\xi)=\mathcal{F}(\Psi (\sum_{j=1}^l \xi^j \varphi_j))
\]
where $\sum_{j=1}^l \xi^j \varphi_j \in W$. We would like to example how well $f$ approximates the functional $\mathcal{F}$ near the origin. Let $u \in U$, and $P_K u \in W$. Note that we have
\begin{align*}
|\mathcal{F}(u)-\mathcal{F}(\Psi(P_K u))| \\ 
=|\int_{0}^1 \frac{d}{ds}\mathcal{F}(u+s(\Psi(P_K u)-u)) ds| \\
= \langle \MF(u+s(\Psi(P_K u) - u)), \Psi(P_K u)-u \rangle_{L^2(\mathbf{V})}
\end{align*}
by the definition of the operator $\MF$ relating the $L^2$ inner product to the given integral. We have by direct computation with the definition of $e_\alpha \cdot \MF$ that
\[
\lVert \MF(u+s(\Psi(P_K u)-u)) - \MF(u) \rVert_{L^2} \leq C \lVert \Psi(P_K u) - u \rVert_{W^{2,2}}^2
\]
lastly, employing the fact from before that $\lVert \Psi(P_K u)-u \rVert_{W^{2,2}} \leq \lVert \MF u \rVert_{L^2}$, we obtain
\begin{align}
|\mathcal{F}(u)-\mathcal{F}(\Psi(P_K u))| \leq C \lVert \MF u \rVert_{L^2}^2
\end{align}
\section{Applying the Łojasiewicz Inequalities to $\mathcal{F}$}
Consider an arbitrary functional $\mathcal{F}$, such that $\mathcal{F}=\int_{\Sigma} F(\omega, u, \nabla^{\mathbf{V}} u)$, where $F=F(\omega, Q)$ is $C^\infty$ and real-valued, $\omega \in \Sigma$, $Q \in \mathbb{R}^{p_1} \times \mathbb{R}^{p_1 p_2}$, and $|Q| < \sigma_0$ for some given $\sigma_0>0$. We establish the conditions for the real analyticity of $\mathcal{F}$:
\subsection{Real-Analyticity Conditions}
For a given $\mathcal{F}$, we assume that $F=F(\omega, z, \eta)$, where $w \in \Sigma$, $z \in \mathbb{R}^{p_2}$, $\eta \in \R^{p_1 p_2}$ is smooth, and that all derivatives with respect to $\omega$ up to order 3 are real-analytic functions of $z, \eta$.
From this we have our first condition: assume that for each $(z_0, \eta_0) \in \mathbb{R}^{p_2} \times \mathbb{R}^{p_1 p_2}$ and for each $j=0,1,2,3$, there exists $C^\infty$ functions $\{a_{\alpha \beta} \}$ on $\Sigma$ correspending to the arbitrary multi-indicies $\alpha=(\alpha_1,...,\alpha_{p_1})$, $\beta=(\beta_1,...,\beta_{p_1 p_2})$ and $\sigma>0$ such that
\[
\sum_{m=0}^{\infty} \left( \sum_{\substack{\alpha + \beta = m}} \sup_{\omega \in \Sigma} \left| D^\alpha a_{\alpha \beta}(\omega) \right|^m \right) < \infty
\]
and
\[
F(\omega, z, \eta) = \sum_{m=0}^{\infty} \sum_{\substack{|\alpha| + |\beta| = m}} a_{\alpha \beta}(\omega) (z - z_0)^\alpha (\eta - \eta_0)^\beta
\]
for $|z-z_0|+|\eta-\eta_0|<\sigma$.
\subsection{The Inequality for $C^3(\mathbf{V})$}
Referencing Section 4.1, we can apply the implicit function theorem given the conditions above on the complexified spaces $\mathbb{C} \otimes C^{2,\alpha}(\mathbf{V})$ and $\mathbb{C} \otimes C^{0,\alpha}(\mathbf{V})$, and so $\Psi=\mathcal{N}^{-1}$ is defined on a neighborhood $U_{\mathbb{C}}$ of $0$ and $C^1$, meaning that for any fixed $u_j \in C^{0,\alpha}(\mathbf{V})$, $j=1,...,R$, the complex derivatives $\frac{\partial}{\partial z^k} (\sum_{j=1}^R z^j u_j)$ are all defined as continuous maps from $U_\mathbb{C}$ into $C_{\mathbb{C}}^{2,\alpha}(\mathbf{V})$. Specifically, the function
\[
f_{\mathbb{C}}(z^1,...,z^n) := \mathcal{F}(\Psi(\sum_{j=1}^l z^j \varphi_j))
\]
is holomorphic in the variable $z=(z^1,...,z^n)$ in some neighborhood of 0 in $\mathbb{C}^l$. Therefore, the real-valued function $f(\xi)=\mathcal{F}(\Psi(\sum_{j=1}^l \xi^j \varphi_j))$ is real-analytic in some neighborhood 0 in $\R^l$. Hence, we can apply the second Łojasiewicz inequality with constants $\alpha \in (0,1]$ and $C,\sigma > 0$ to obtain
\begin{align}
|f(\xi)|^{1-\frac{\alpha}{2}} \leq C|\nabla f(\xi)| \hspace{0.5cm} \forall \xi \in B_\sigma(0)
\end{align}
where $B_\sigma(0)$ is a ball in $\R^l$ of radius $\sigma$ centered at the origin. Now by (4) of Section 4.4, with $u \in U$ chosen such that $\xi^j=\langle u,\varphi \rangle_{L^2}$ satisfy $|\xi| < \sigma$, we have that
\[
|\mathcal{F}(u)-f(\xi)| \leq C \lVert \MF(u) \rVert_{L^2}^2
\]
and so by (5) and (3) of Section 4.3, we have
\[
|\mathcal{F}(u)|^{1-\alpha/2} \leq C(\lVert \MF(u) \rVert_{L^2} + \lVert \MF(u) \rVert_{L^2}^{2-\alpha}) \leq C(\lVert \MF(u) \rVert_{L^2})
\]
for each $u \in U$ such that $\lVert P_K(u) \rVert < \sigma$. Formally, $\exists \sigma > 0$ such that 
\[
|\mathcal{F}(u)|^{1-\alpha/2} \leq C \lVert \MF(u) \rVert_{L^2} \hspace{0.5cm} \forall u \in C^3(\mathbf{V}), \lVert u \rVert_{C^3} < \sigma_0
\]
which is the Łojasiewicz inequality we desired to show for $C^3(\mathbf{V})$.
\subsection{Non-Real-Analytic Functionals}
Now would would like to examine the conditions for which the inequality for $C^3(\mathbf{V})$ holds with best exponent $(0,1]$ without any real-analyticity hypothesis. For this, we state one simple theorem pertaining to this condition:
\begin{theorem}{5.3.1}
    Given $\mathcal{F} \in C^\infty(C^1(\mathbf{V}); \R)$ but not real-analytic, if we assume the "integrability condition"
    \[
    \mathcal{F}(\Psi(P_K(u))) \equiv 0 \hspace{0.2cm} \text{on some} \hspace{0.2cm} C^3\text{-neighborhood of} \hspace{0.2cm} 0 \hspace{0.2cm} \text{in}  \hspace{0.2cm} C^3(\mathbf{V})
    \]
    then indeed,
    \[
    |\mathcal{F}(u)|^{1/2} \leq C \lVert \MF(u) \rVert_{L^2} \hspace{0.5cm} \forall u \in C^3(\mathbf{V}), \lVert u \rVert_{C^3} < \sigma_0
    \]
    \textbf{Remark:} the integrability condition is equivalent to the assuption that there exists a $C^\infty$ $l$-dimensional manifold of solutions of the non-linear equation $\MF(u)=0$ which is tangent $K=\ker(\LF)$ at 0.
    \begin{proof}
        Assuming that $\mathcal{F}(\Psi(P_K(u))) \equiv 0$ on some $C^3$ neighborhood of 0 in $C^3(\mathbf{V})$ and appealing to (4) of Section 4.3, we have that 
        \begin{align*}
        |\mathcal{F}(u)|=|\mathcal{F}(u)-\mathcal{F}(\Psi(P_K(u)))| \leq C \lVert \MF(u) \rVert_{L^2}^2 \\
        \implies |\mathcal{F}(u)|^{1/2} \leq C \lVert \MF(u) \rVert_{L^2}
        \end{align*}
        as required.
    \end{proof}
\end{theorem}
\section{The Łojasiewicz Inequality for $\mathcal{E}_{S^{n-1}}$}
This is where the theory of Łojasiewicz Inequalities on arbitrary real-analytic functions on vector bundles crosses over with the study of energy minimizing maps. Recall the following definition of the energy functional, yet for maps $u \in C^\infty(S^{n-1}; N)$ is now given by
\[
\mathcal{E}_{S^{n-1}}(u)=\int_{S^{n-1}} |Du|^2
\]
However, for the energy functional applied to the sphere, the smooth maps $C^\infty(S^{n-1}; N)$ do not form a linear space. Therefore, we must show that for maps that are $C^3$ close to a given harmonic map $\varphi_0 \in C^\infty(S^{n-1};N)$, we can write $\mathcal{E}$ as a functional over a vector bundle.
\begin{figure}[H] 
    \centering
    \includegraphics[width=0.75\linewidth]{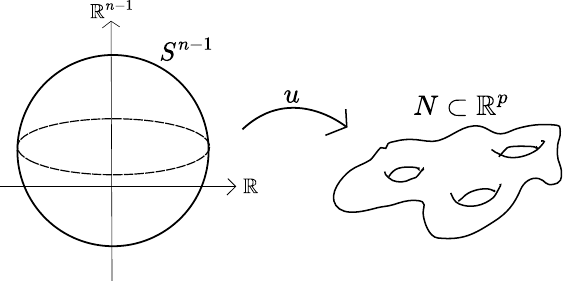}
    \caption{\textit{The action of} $u \in C^\infty(S^{n-1}; N)$}
    \label{fig:my_label}
\end{figure}
\subsection{Defining the Functional $\mathcal{F}$}
Consider a fixed $\varphi_0 \in C^\infty(S^{n-1};N)$ which is a harmonic map. By Section 2.1, we have the following result:
\[
\varphi_0 \hspace{0.2cm} \text{is harmonic} \iff \Delta_{s^{n-1}}\varphi_0 + A_{\varphi_0}(D_\omega \varphi_0, D_\omega \varphi_0) = 0
\]
\textbf{Note:} $A_{\varphi_0}(D_\omega \varphi_0, D_\omega \varphi_0)$ is shorthand for $\sum_{j=1}^{n-1} A_{\varphi_0} (\nabla_{\tau_i} \varphi_0, \nabla_{\tau_i} \varphi_0)$, where $\tau_1,...\tau_{n-1}$ is an orthonormal basis for $T_{\omega} S^{n-1}$. Before continuing on with further definitions, the crux of this chapter can be stated as follows:
\begin{theorem}{6.1.1}
    For maps $u \in C^{\infty}(S^{n-1}; N)$ such that $\forall \epsilon > 0$
    \[
    \lVert u - \varphi_0 \rVert_{C^3}
    \]
    for a given harmonic map $\varphi_0 \in C^{\infty}(S^{n-1}; N)$, the energy functional $\mathcal{E}_{S^{n-1}}$ can be expressed as a functional $\mathcal{F}$ over a vector bundle of a compact $C^1$ Riemannian manifold.
\end{theorem}
\begin{proof}
    Assume $\delta>0$ and $y_0 \in N$, and let
    \begin{align*}
        T_{\delta}(y_0)=\{\tau \in T_{y_0}N: |\tau| < \delta \} \\
        U_{\delta}(y_0)=\{\Pi(y_0 + \tau): \tau \in T_{\delta}(y_0) \}
    \end{align*}
    where $\Pi$ is the nearest point projection Then for $\delta$ depending only on the target manifold $N$, $U_{\delta}(y_0)$ is a neighborhood of $y_0$ in $N$. Moreover, since $\Pi$ is a smooth (real-analytic of $N$ is real analytic) map, then the mapping given by 
    \[
        \Phi_{y_0}^{-1}(y): \tau \rightarrow \Pi(y_0 + \tau)
    \]
    is a smooth diffeomorphism of $T_\delta (y_0)$ onto $U_{\delta}(y_0)$. Further, $\Phi_{y_0}^{-1}(y)$ depends smoothly or real analytically on $(y_0,y) \in N \times N$, contingent on whether $N$ is real-analytic or merely $C^\infty$. Now define the vector bundle 
    \[
    \mathbf{V}=\varphi_0^* TN \equiv \{T_{\varphi_0(\omega)}N \}_{\omega \in S^{n-1}}
    \]
    Notice if $u \in C^2(S^{n-1}; N)$ with $\lVert u-\varphi_0 \rVert_{C^2} < \delta$, the we can leverage the definition of $\Phi_{y_0}$ to write for $\omega \in S^{n-1}$
    \begin{align*}
    u(\omega)=\Pi(\varphi_0 + \Phi_{\varphi_0(\omega)}^{-1}(u(\omega))) \\
    =\Pi(\varphi_0 + \tilde{u})
    \end{align*}
    where $\tilde{u}(\omega)=\Phi_{\varphi_0(\omega)}^{-1}(u(\omega)) \in C^2(\mathbf{V})$. Furthermore, we have that
    \[
    \mathcal{E}_{S^{n-1}}(u)=\mathcal{E}_{S^{n-1}}(\Pi(\varphi_0+\tilde{u}))
    \]
    Now the expression $\mathcal{E}_{S^{n-1}}(\Pi(\varphi_0+\tilde{u}))$ takes the form of a functional \[
    \int_{S^{n-1}} F(\omega, \tilde{u}, \nabla^{\mathbf{V}} \tilde{u})
    \]
    where $\mathbf{V}=\{V \}_{\omega \in S^{n-1}}$, $V_\omega = T_{\varphi_0(\omega)}N$ and
    \[
    F=F(\omega,z,\eta), \hspace{0.25cm} \omega \in S^{n-1}, z \in \mathbb{R}^p, \eta \in \mathbb{R}^{np}
    \]
    Now by using the definition 3.1.1 for the gradient operator on a vector bundle $\mathbf{V}$, and by relating $\mathcal{E}_{S^{n-1}}(\Pi(\phi_0+\tilde{u}))$ to the functional form $\int_{S^{n-1}} F(\omega, \tilde{u}, \nabla^{\mathbf{V}} \tilde{u})$, we obtain that
    \[
    F(\omega, z, \eta)=|d\Pi_{\varphi_0(\omega)+z}((\sum_{\alpha=1}^p \nabla^{S^{n-1}} \varphi_0^\alpha (\omega) \otimes P_\omega e_\alpha)+\eta)|^2
    \]
    Considering that $\Pi$ is smooth (and real-analytic is $N$) is real analytic, then $F$ is a $C^\infty$ function of $\omega, z, \eta$ for $|z|<\delta$. Now we would like to make the following definition for the functional $\mathcal{F}$:
    \[
    \mathcal{F}(u)=\int_{S^{n-1}}(F(\omega,u,\nabla^{\mathbf{V}})-F(\omega,0,0))
    \]
    Notice that
    \[
    \mathcal{E}_{S^{n-1}}(u)-\mathcal{E}_{S^{n-1}}(\varphi_0)=\mathcal{F}(\tilde{u})=\mathcal{E}_{S^{n-1}}(\Pi(\varphi_0+\tilde{u}))-\mathcal{E}_{S^{n-1}}(\varphi_0)
    \]
    for $\tilde{u}$ such that $u=\Pi(\varphi_0+\tilde{u})$. In other words, for a function $v \in C^2(\mathbf{V})$ such that $\lVert v \rVert_{C^2} \leq \delta$, we conclude
    \[
    \mathcal{F}(v)=\mathcal{E}_{S^{n-1}}(\Pi(\varphi_0+v))-\mathcal{E}_{S^{n-1}}(\varphi_0)
    \]
\end{proof}
\subsection{The Euler-Lagrange Multiplier for $\mathcal{E}_{S^{n-1}}(u)$}
Referencing Section 2.1 on Variational Equations, we that for the 1-parameter family $\{u_s\}_{s \in (-\delta, \delta)}$ of maps of $B_\rho(y)$ into $N$ such that $u_0=u$, $Du_s \in L^2(\Omega)$, and $u_s \equiv u$ in a neighborhood of $\delta B_\rho (y)$ for each $s \in (-\delta,\delta)$. We observe that for $u$ energy minimizing, $\mathcal{E}_{B_\rho(y)}(u_s)$ attains its minimum at $s=0$, and
\[
{\frac{d \mathcal{E}_{B_\rho(y)}(u_s)}{ds}}|_{s=0}=0
\]
Further on in the discussion, we also obtain the fact that
\[
\Delta u + \sum_{i=1}^n A_u (D_i u, D_i u) = 0
\]
for $u$ energy minimizing. From this, we deduce that while working over the space $S^{n-1}$ and using the correct gradient operator, this condition can be extrapolated to obtain the Euler-Lagrange operator for the energy functional:
\begin{definition}{6.2.1 (The Euler-Lagrange Operator $\mathcal{M}_{\mathcal{E}_{S^{n-1}}}$)}
The Euler-Lagrange operator for the energy functional is exactly
\[
\mathcal{M}_{\mathcal{E}_{S^{n-1}}}=\nabla_{S^{n-1}}u+A_u(D_\omega, D_\omega)
\]
and it satisfies the identity
\[
\frac{d}{ds} \mathcal{E}_{S^{n-1}}(\Pi(u+sv))|_{s=0}=-\langle \mathcal{M}_{\mathcal{E}_{S^{n-1}}}, v \rangle_{L^2}, \hspace{0.25cm} v=(v^1,...,v^p) \in C^2(S^{n-1}; \mathbb{R})
\]
\end{definition}
Now by combining this with the definition for $\mathcal{F}(\tilde{u})$ given previously, and the fact that $\frac{d}{ds} \mathcal{F}(u+sv)|_{s=0}=-\langle \MF(u), v \rangle_{L^2}$, we have the following relation:
\[
\langle \mathcal{M}_{\mathcal{E}_{S^{n-1}}}(u), v \rangle_{L^2} = \langle \MF(\tilde{u}), v \rangle_{L^2}
\]
for $v \in C^2(\mathbf{V})$. Now notice that taking the orthogonal projection onto $T_{\varphi_0}N$ annihilates the second component of the inner product, namely $v \in C^2(\mathbf{V})$, and hence we have
\[
(M_{\epsilon_{s_n-1}}(u))^{T_0} \equiv M_f(u), \quad u \in C^2(S^{n-1}; N), \quad \|u - \phi_0\|_{C} < \delta,
\]
where $(\cdot)^{T_0}$ is the orthogonal projection onto $T_{\varphi_0(\omega)}N$. Using the exact definition for $\mathcal{M}_{\mathcal{E}_{S^{n-1}}}$ as before, we have that 
\[
\mathcal{M}_{\mathcal{E}_{S^{n-1}}}=(\Delta u)^T
\]
and $(\cdot)^T$ means orthogonal projection onto $T_u(\omega)N$. Because we assumed $\lVert u - \varphi_0 \rVert_{C^2} < \delta$, we have the bounds
\[
(1-C\delta)|\mathcal{M}_{\mathcal{E}_{S^{n-1}}}(u)| \leq |\MF(\tilde{u})| \leq |\mathcal{M}_{\mathcal{E}_{S^{n-1}}}(u)|
\]
\subsection{Linearization of $\mathcal{M}_{\mathcal{E}_{S^{n-1}}}$ and Final Estimates}
Recall the identity
\[
\mathcal{F}(v)=\mathcal{E}_{S^{n-1}}(\Pi(\varphi_0+v))-\mathcal{E}_{S^{n-1}}(\varphi_0)
\]
and by taking the mixed partial derivative $\frac{\partial^2}{\partial s \partial t}$ on both sides considering the transformation $v \rightarrow sv+tw$, we have
\[
\mathcal{L}_{\varphi_0}(v) \equiv \mathcal{L}_{\mathcal{F}}(v)
\]
where $\mathcal{L}_{\varphi_0}(v)=\frac{d}{ds}\mathcal{M}_{\mathcal{E}_{S^{n-1}}}(\Pi(\varphi_0+sv))|_{s=0}$ is the linearization of $\mathcal{M}_{\mathcal{E}_{S^{n-1}}}$ at $\varphi_0$ and $\LF(v)=\frac{d}{ds} \MF(sv)|_{s=0}$ is the linearization of $\MF$ at 0. It is easy to verify that $\LF$ is elliptic, and as it has the second orfer term $\Delta v$, we can apply the theory set out in this paper to the functional $\mathcal{F}$. Assuming that $N$ is real-analytic, which renders $\mathcal{F}$ real-analytic, we can apply the Łojasiewicz inequality for arbitrary real-analytic functions in Section 4.2 to obtain
\begin{align*}
    |\mathcal{F}(\tilde{u})|^{1-\alpha/2} \leq C \lVert \mathcal{M}_{\mathcal{F}} (\tilde{u}) \rVert_{L^2} \\
    \implies |\mathcal{E}_{S^{n-1}}(u)-\mathcal{E}_{S^{n-1}}(\varphi_0)|^{1-\alpha/2} \leq C \lVert \mathcal{M}_{\mathcal{F}} (\tilde{u}) \rVert_{L^2} \\ \leq C \lVert \mathcal{M}_{\mathcal{E}_{S^{n-1}}} (u) \rVert_{L^2}
\end{align*}
for $u \in C^\infty(S^{n-1}; N)$ with $\lVert u-\varphi_0 \rVert_{C^3} < \sigma$. 
In the discussion of Section 4.3, we also discussed the case where the target manifold $N$, and by extension $\mathcal{F}$, are merely smooth. In this case, assuming that $\mathcal{F}$ still satisfies the "integrability condition"
\[
    \mathcal{F}(\Psi(P_K(u))) \equiv 0 \hspace{0.2cm} \text{on some} \hspace{0.2cm} C^3\text{-neighborhood of} \hspace{0.2cm} 0 \hspace{0.2cm} \text{in}  \hspace{0.2cm} C^3(\mathbf{V})
\]
then Łojasiewicz with best exponent $\alpha=1$ holds, i.e. in this particular instance, there are $C, \sigma > 0$ such that 
\[
|\mathcal{E}_{S^{n-1}}(u)-\mathcal{E}_{S^{n-1}}(\varphi_0)|^{1/2} \leq C \lVert \mathcal{M}_{\mathcal{E}_{S^{n-1}}}(u) \rVert_{L^2}
\]
for $u \in C^\infty(S^{n-1};N)$ s.t. $\lVert u - \varphi_0 \rVert_{C^3} < \sigma$.

\section{Open Problems}
Here we present three open problems pertaining to application of Simon-Łojasiewicz inequalities over vector bundles:

\begin{op}{(Functionals on Sections)}
Is there a formulation for a global version of the Simon-Łojasiewicz inequalities for functionals defined on sections of vector bundles over compact Riemannian Manifolds? Specifically, given a manifold $\Sigma$ as defined in Section 3 and $\Gamma(\Sigma)$ the space of smooth sections of $\Sigma$ and $\mathcal{F}: \Gamma(\Sigma) \rightarrow \mathbb{R}$, determine the conditions for which there is a global inequality of the form
\[
|\mathcal{F}(\sigma)-\mathcal{F}(\sigma_0)|^{\alpha} \leq C|\nabla \mathcal{F}(\sigma)|
\]
for $\sigma \in \Gamma(\Sigma)$, $\sigma_0$ is a critical point of $\mathcal{F}$, $C>0$ and $\alpha \in (0,1]$. What do $\alpha$ and $C$ depend on?
\end{op}
\begin{op}{(Convergence of Gradient Flows)}
Determine rates of convergence for the gradient flow of functionals on vector bundles using the Simon-Łojasiewicz inequalities. Formally, assume $\mathcal{F}$ is a real-analytic functional on a vector bundle $\mathbf{V}$ over a manifold $\Sigma$. Consider the gradient flow given by
\[
\frac{d}{dt}s(t)=-\nabla\mathcal{F}(s(t))
\]
For any solution $s(t)$, does there exist a rate function $R(t)$ such that
\[
\lVert s(t)-s_\infty \rVert \leq R(t)
\]
where $s_\infty$ is the a limit point of the flow? What conditions on $\mathcal{F}$ are need to ensure the existence of $s_\infty$? How does $R(t)$ depend on the 
Łojasiewicz exponent $\alpha$? 
\end{op}
\begin{op}{(Sobolev Functionals)}
Generalize the Łojasiewicz-Simon inequalities to functionals that are not necessarily real-analytic, but belong to a given Sobolev space $W^{k,p}$ and are $C^k$. Formally, let $\mathcal{F}: W^{k,p}(\mathbf{V}) \rightarrow \mathbb{R}$ be a $C^k$ functional, $\mathbf{V}$ a vector bundle of a compact manifold $\Sigma$. Determine under what conditions (if any) there is a Simon-Łojasiewicz inequality of the form
\[
|\mathcal{F}(u)-\mathcal{F}(u_0)|^\alpha \leq C \lVert D\mathcal{F}(u) \rVert_{W^{k-1,p}}
\]
where $u \in W^{k,p}$, $u_0$ is a critical point, and $D\mathcal{F}$ is the Fréchet derivative of $\mathcal{F}$. How are $\alpha$ and $C$ determined?
\end{op}

\end{document}